\documentclass[10pt, a4paper]{article}

\usepackage{color,url,latexsym,amssymb,amsthm,amsmath,amsxtra,amsfonts}
\usepackage{graphicx}
\usepackage{amsfonts}
\usepackage{amsbsy}
\usepackage{amssymb}
\usepackage{amsmath}
\usepackage{amsthm}
\theoremstyle{plain}
\newtheorem{theorem}{Theorem}[section]
\newtheorem{lemma}{Lemma}[section]
\newtheorem{corollary}{Corollary}[section]
\newtheorem{proposition}{Proposition}[section]
\theoremstyle{definition}

\newtheorem{remark}{Remark}[section]

\begin{document}

\title{Vertical Liouville foliations on the big-tangent manifold of a Finsler space}
\author{Cristian Ida and Paul Popescu}
\date{}
\maketitle

\begin{abstract}
The present paper unifies some aspects concerning the vertical Liouville distributions on the tangent (cotangent) bundle of a Finsler (Cartan) space in the context of generalized geometry. More exactly, we consider the big-tangent manifold $\mathcal{T}M$ associated to a Finsler space $(M,F)$ and of its $\mathcal{L}$-dual which is a Cartan space $(M,K)$ and we define three Liouville distributions on $\mathcal{T}M$ which are integrable. We also  find geometric properties of both leaves of Liouville distribution and the vertical distribution in our context.
\end{abstract}

\medskip

\begin{flushleft}
\strut \textbf{2010 Mathematics Subject Classification}: 53B40, 53C12, 53C60.

\end{flushleft}

\section{Introduction and preliminary notions}
\setcounter{equation}{0}

\subsection{Introduction}
The vertical Liouville distribution on the tangent bundle of a (pseudo) Finsler space was defined for the first time in \cite{B-F1} where some aspects of the geometry of the vertical bundle are derived via  vertical Liouville distribution. A similar study on the cotangent bundle of a Cartan space can be found in \cite{I-M}. Also, other signifiant studies concerning the interrelations between natural foliations defined by Liouville fields on the tangent bundle of a Finsler space and the geometry of the Finsler space itself, as well as similar problems on Cartan spaces are intensively studied in \cite{Bej} and \cite{A-R1}, respectively. See also \cite{I-M, M-I, P-T-Z, P-N}.

As it is well known, in the \textit{generalized geometry} intitiated in \cite{hi}, the tangent bundle $TM$ of a smooth $n$-dimensional manifold $M$ is replaced by the \textit{big-tangent bundle} (or Pontryagin bundle) $TM\oplus T^*M$. On its total space the velocities and momenta are considered as independent variables. This idea was proposed and developed in \cite{S1, S2} and later was used in the study of Hamiltonian-Jacoby theory for singular Lagrangian systems \cite{L-D-V}. The geometry of the total space of the big-tangent bundle, called \textit{big-tangent manifold}, is intensively studied in \cite{Va13} and some its applications to mechanical systems can be found in \cite{G}.

Using the framework of the geometry on the big-tangent manifold, our aim in this paper is to extend some results concerning the vertical Liouville foliation in the context of generalized geometry. In this sense, we consider the big-tangent manifold $\mathcal{T}M$ associated to a Finsler space $(M,F)$ and of its $\mathcal{L}$-dual which is a Cartan space $(M,K)$.   As usual, we reconsider the vertical Liouville distributions $V_{\mathcal{E}_1}$ and $V_{\mathcal{E}_2}$ from the case of vertical tangent (cotangent) bundle of a Finsler (Cartan) space, see \cite{B-F1, I-M}, for the case of vertical subbundles $V_1$ and $V_2$, respectively, with respect to Liouville vector fields $\mathcal{E}_1$ and $\mathcal{E}_2$. Next we define the Liouville distribution $V_{\mathcal{E}}$ with respect to the  Liouville vector field $\mathcal{E}=\mathcal{E}_1+\mathcal{E}_2$, we prove that it is integrable (Theorem \ref{integrabil}) and we study some of its properties (Theorems \ref{t2} and \ref{t3}). Also, some links between the vertical Liouville foliations $V_{\mathcal{E}_1}$, $V_{\mathcal{E}_2}$ and $V_{\mathcal{E}}$, respectively, are established.

\subsection{Preliminaries and notations}

Let $M$ be a $n$-dimensional smooth manifold, and we consider $\pi:TM\rightarrow M$ its tangent bundle, $\pi^*:T^*M\rightarrow M$ its cotangent bundle and $\tau\equiv \pi\oplus\pi^*:TM\oplus T^*M\rightarrow M$ its big-tangent bundle defined as Whitney sum of the tangent and cotangent bundles of $M$. The total space of the big-tangent bundle, called \textit{big-tangent manifold}, is a $3n$-dimensional smooth manifold denoted here by $\mathcal{T}M$. Let us briefly recall some elementary notions about the big-tangent manifold $\mathcal{T}M$. For a detalied discussion about its geometry we refer \cite{Va13}. 

Let $(U,(x^{i}))$ be a local chart on $M$. If $\{\frac{\partial}{\partial x^{i}}|_x\}$, $x\in U$ is a local frame of sections in the tangent bundle over $U$ and $\{dx^{i}|_x\}$, $x\in U$ is a local frame of sections in the cotangent bundle over $U$, then by definition of the Whitney sum, $\{\frac{\partial}{\partial x^{i}}|_x,dx^{i}|_x\}$, $x\in U$ is a local frame of sections in the big-tangent bundle $TM\oplus T^*M$ over $U$. Every section $(y,p)$ of $\tau$ over $U$ takes the form $(y,p)=y^{i}\frac{\partial}{\partial x^{i}}+p_i dx^{i}$ and the local coordinates on $\tau^{-1}(U)$ will be defined as the triples $(x^{i},y^{i},p_i)$,  where $i=1,\ldots,n=\dim M$, $(x^{i})$ are local coordinates on $M$, $(y^{i})$ are vector coordinates and $(p_i)$ are covector coordinates. 

The change rules of these coordinates are:
\begin{equation}
\label{I1}
\widetilde{x}^{i}=\widetilde{x}^{i}(x^j)\,,\,\widetilde{y}^{i}=\frac{\partial\widetilde{x}^{i}}{\partial x^j}y^j\,,\,\widetilde{p}_i=\frac{\partial x^j}{\partial\widetilde{x}^{i}}p_j
\end{equation} 
and the local expressions of a vector field $X$ and of a $1$-form $\varphi$ on $\mathcal{T}M$ are
\begin{equation}
X=\xi^{i}\frac{\partial}{\partial x^{i}}+\eta^{i}\frac{\partial}{\partial y^{i}}+\zeta_i\frac{\partial}{\partial p_i}\,\,\,{\rm and}\,\,\,\varphi=\alpha_idx^{i}+\beta_idy^{i}+\gamma^{i}dp_i.
\label{I2}
\end{equation}
For the big-tangent manifold $\mathcal{T}M$ we have the following projections
\begin{displaymath}
\tau:\mathcal{T}M\rightarrow M\,,\,\tau_1:\mathcal{T}M\rightarrow TM\,,\,\tau_2:\mathcal{T}M\rightarrow T^*M
\end{displaymath}
on $M$ and on the total spaces of tangent and cotangent bundle, respectively. 

As usual, we denote by $V=V(\mathcal{T}M)$ the vertical bundle on the big-tangent manifold $\mathcal{T}M$ and it has the decomposition
\begin{equation}
\label{I4}
V=V_1\oplus V_2,
\end{equation}
where $V_1=\tau_1^{-1}(V(TM))$, $V_2=\tau_2^{-1}(V(T^*M))$ and have the local frames $\{\frac{\partial}{\partial y^{i}}\}$, $\{\frac{\partial}{\partial p_i}\}$, respectively. The subbundles $V_1$, $V_2$ are the vertical foliations of $\mathcal{T}M$ by fibers of $\tau_1, \tau_2$, respectively, and $\mathcal{T}M$ has a multi-foliate structure \cite{Va70}. The \textit{Liouville vector fields} (or Euler vector fields) are given by
\begin{equation}
\label{I5}
\mathcal{E}_1=y^{i}\frac{\partial}{\partial y^{i}}\in\Gamma(V_1)\,,\,\mathcal{E}_2=p_i\frac{\partial}{\partial p_i}\in\Gamma(V_2)\,,\,\mathcal{E}=\mathcal{E}_1+\mathcal{E}_2\in\Gamma(V).
\end{equation}

In the following we consider that manifold $M$ is endowed with a Finsler structure $F$, and we present a metric structure on $V$ induced by $F$. According to \cite{B-C-S, B-F, M-A}, a function $F:TM\rightarrow[0,\infty)$ which satisfies the following conditions:
\begin{enumerate}
\item[i)] $F$ is $C^\infty$ on $TM^0=TM-\{{\rm zero\,\,section}\}$;
\item[ii)] $F(x,\lambda y)=\lambda F(x,y)$ for all $\lambda\in\mathbb{R}_+$;
\item[iii)] the $n\times n$ matrix $(g_{ij})$, where $g_{ij}=\frac{1}{2}\frac{\partial^2F^2}{\partial y^{i}\partial y^j}$, is positive definite at all points of $TM^0$,
\end{enumerate}
is called a \textit{Finsler structure} on $M$ and the pair $(M,F)$ is called a \textit{Finsler space}. We notice that in fact $F(x,y)>0$, whenever $y\neq0$.

There are some useful facts which follow from the above  homogeneity condition ii) of the fundamental function of the Finsler space $(M,F)$. By the Euler theorem on positively homogeneous functions we have, see \cite{B-C-S, B-F, M-A}:
\begin{equation}
y_i=g_{ij}y^j\,,\,y^{i}=g^{ij}y_j\,,\,F^2=g_{ij}y^{i}y^j=y_iy^i\,,\,C_{ijk}y^k=C_{ikj}y^k=C_{kij}y^k=0,
\label{I6}
\end{equation}
where $(g^{ij})$ is the inverse matrix of $(g_{ji})$ and we have put $y_i=\frac{1}{2}\frac{\partial F^2}{\partial y^i}\,,\,C_{ijk}=\frac{1}{4}\frac{\partial^3F^2}{\partial y^i\partial y^j\partial y^k}$.

Also, for a Finsler structure $F$ on $TM^0$ there is Cartan structure $K=F^*$ on $T^*M^0:=T^*M-\{{\rm zero\,\,section}\}$ obtained by Legendre transformation of $F$ (the $\mathcal{L}$-duality proces, see \cite{M1, M2, M-H-S}), that is a function $K:T^*M\rightarrow[0,\infty)$ which has the following properties:
\begin{enumerate}
\item[i)] $K$ is $C^{\infty}$ on $T^*M^0$;
\item[ii)] $K(x,\lambda p)=\lambda K(x,p)$ for all $\lambda>0$;
\item[iii)] the $n\times n$ matrix $(g^{*ij})$, where $g^{*ij}=\frac{1}{2}\frac{\partial^2 K^2}{\partial p_i\partial p_j}$, is positive definite at all points of $T^*M_0$. 
\end{enumerate}
Also $K(x,p)>0$, whenever $p\neq0$. The properties of $K$ imply that
\begin{equation}
p^i=g^{*ij}p_j\,,\,p_i=g^*_{ij}p^j\,,\,K^2=g^{*ij}p_ip_j=p_ip^i\,,\,C^{ijk}p_k=C^{ikj}p_k=C^{kij}p_k=0,
\label{I7}
\end{equation}
where $(g^*_{ij})$ is the inverse matrix of $(g^{*ji})$ and we have put $p^i=\frac{1}{2}\frac{\partial K^2}{\partial p_i}\,,\,C^{ijk}=-\frac{1}{4}\frac{\partial^3K^2}{\partial p_i\partial p_j\partial p_k}$.

It is well-known that $g_{ij}$ determines a metric structure on $V(TM)$ and $g^{*ij}$ determines a metric structure on $V(T^*M)$. Similarly, every Finsler structure $F$ on $M$ determines a metric structure $G$ on $V$ by setting
\begin{equation}
\label{I8}
G(X,Y)=g_{ij}(x,y)X_1^{i}(x,y,p)Y_1^j(x,y,p)+g^{*ij}(x,p)X^2_i(x,y,p)Y^2_j(x,y,p),
\end{equation}
for every $X=X_1^{i}(x,y,p)\frac{\partial}{\partial y^{i}}+X^2_i(x,y,p)\frac{\partial}{\partial p_i}$, $Y=Y_1^{j}(x,y,p)\frac{\partial}{\partial y^{j}}+Y^2_j(x,y,p)\frac{\partial}{\partial p_j}\in\Gamma(V)$.

\section{Vertical Liouville foliations on $\mathcal{T}M$}
\setcounter{equation}{0}

In this section we reconsider the vertical Liouville distributions $V_{\mathcal{E}_1}$ and $V_{\mathcal{E}_2}$ from the case of vertical tangent (cotangent) bundle of a Finsler (Cartan) space, see \cite{B-F1, I-M}, for the case of vertical subbundles $V_1$ and $V_2$, respectively, with respect to Liouville vector fields $\mathcal{E}_1$ and $\mathcal{E}_2$. Next we define the Liouville distribution $V_{\mathcal{E}}$ with respect to the  Liouville vector field $\mathcal{E}=\mathcal{E}_1+\mathcal{E}_2$, we prove that it is integrable and we study some of its properties. Also, some links between the vertical Liouville foliations $V_{\mathcal{E}_1}$, $V_{\mathcal{E}_2}$ and $V_{\mathcal{E}}$, respectively, are established.

\subsection{Vertical Liouville distributions $V_{\mathcal{E}_1}$ and $V_{\mathcal{E}_2}$}

Following \cite{B-F1}, \cite{I-M} we define two vertical Liouville distributions on $\mathcal{T}M$ as the complementary orthogonal distributions in $V_1$ and $V_2$  to the line distributions spanned by the  Liouville vector fields $\mathcal{E}_1$ and $\mathcal{E}_2$, respectively. 

By \eqref{I5} and \eqref{I6} we have
\begin{equation}
G(\mathcal{E}_1,\mathcal{E}_1)=F^2.
\label{xII1}
\end{equation}
Using $G$ and $\mathcal{E}_1$, we define the $V_1$-vertical one form $\zeta_1$ by
\begin{equation}
\zeta_1(X_1)=\frac{1}{F}G(X_1,\mathcal{E}_1)\,,\,\forall\,X_1=X_1^{i}(x,y,p)\frac{\partial}{\partial y^{i}}\in\Gamma(V_1).
\label{xII2}
\end{equation}
Let us denote by $\left\{\mathcal{E}_1\right\}$ the line vector bundle over $\mathcal{T}M$ spanned by $\mathcal{E}_1$ and we define the \textit{first vertical Liouville distribution} as the complementary orthogonal distribution $V_{\mathcal{E}_1}$ to $\left\{\mathcal{E}_1\right\}$ in $V_1$ with respect to $G$. Thus, $V_{\mathcal{E}_1}$ is defined by $\zeta_1$, that is
\begin{equation}
\Gamma\left(V_{\mathcal{E}_1}\right)=\{X_1\in\Gamma(V_1)\,:\,\zeta_1(X_1)=0\}.
\label{xII3}
\end{equation}
We get that every $V_1$-vertical vector field $X_1=X_1^{i}(x,y,p)\frac{\partial}{\partial y^{i}}$ can be expressed in the form:
\begin{equation}
X_1=P_1X_1+\frac{1}{F}\zeta_1(X_1)\mathcal{E}_1,
\label{xII4}
\end{equation}
where $P_1$ is the projection morphism of $V_1$ on $V_{\mathcal{E}_1}$. 

Also, by direct calculus, we get
\begin{equation}
G(X_1,P_1Y_1)=G(P_1X_1,P_1Y_1)=G(X_1,Y_1)-\zeta_1(X_1)\zeta_1(Y_1),\,\,\forall\,X_1,Y_1\in\Gamma(V_1).
\label{xII5}
\end{equation}
Let us consider $\{\theta^{i}\}$ the dual basis of $\{\frac{\partial}{\partial y^{i}}\}$. Then, with respect  to the basis $\{\theta^{i}\}$ and $\left\{\theta^{j}\otimes\frac{\partial}{\partial y^i}\right\}$, respectively, $\zeta_1$ and $P_1$ are locally given by
\begin{equation}
\zeta_1=\stackrel{1}{\zeta_{i}}\theta^{i}\,,\,P_1=\stackrel{1}{P^{i}_j}\theta^{j}\otimes\frac{\partial}{\partial y^{i}}\,,\,\stackrel{1}{\zeta_i}=\frac{y_i}{F}\,,\,\stackrel{1}{P^{i}_j}=\delta^{i}_j-\frac{y_jy^{i}}{F^2},
\label{xII6}
\end{equation}
where $\delta_j^i$ are the components of the Kronecker delta.

As usual for tangent bundle of a Finsler space (see Theorem 3.1 from \cite{B-F1}), the first vertical Liouville distribution $V_{\mathcal{E}_1}$ is integrable and it defines a foliation on $\mathcal{T}M$,  called the \textit{first vertical Liouville foliation} on the big-tangent manifold $\mathcal{T}M$. Also, some geometric properties of the leaves of vertical foliation $V_1$ can be derived via the first vertical Liouville foliation $V_{\mathcal{E}_1}$.

Similarly, by \eqref{I5} and \eqref{I7} we have
\begin{equation}
G(\mathcal{E}_2,\mathcal{E}_2)=K^2,
\label{yII1}
\end{equation}
and using $G$ and $\mathcal{E}_2$, we define the $V_2$-vertical one form $\zeta_2$ by
\begin{equation}
\zeta_2(X_2)=\frac{1}{K}G(X_2,\mathcal{E}_2)\,,\,\forall\,X_2=X^2_{i}(x,y,p)\frac{\partial}{\partial p_{i}}\in\Gamma(V_2).
\label{yII2}
\end{equation}
Let us denote by $\left\{\mathcal{E}_2\right\}$ the line vector bundle over $\mathcal{T}M$ spanned by $\mathcal{E}_2$ and we define the \textit{second vertical Liouville distribution} as the complementary orthogonal distribution $V_{\mathcal{E}_2}$ to $\left\{\mathcal{E}_2\right\}$ in $V_2$ with respect to $G$. Thus, $V_{\mathcal{E}_2}$ is defined by $\zeta_2$, that is
\begin{equation}
\Gamma\left(V_{\mathcal{E}_2}\right)=\{X_2\in\Gamma(V_2)\,:\,\zeta_2(X_2)=0\}.
\label{yII3}
\end{equation}
We get that every $V_2$-vertical vector field $X_2=X^2_{i}(x,y,p)\frac{\partial}{\partial p_{i}}$ can be expressed in the form:
\begin{equation}
X_2=P_2X_2+\frac{1}{K}\zeta_2(X_2)\mathcal{E}_2,
\label{yII4}
\end{equation}
where $P_2$ is the projection morphism of $V_2$ on $V_{\mathcal{E}_2}$. 

Similarly, by direct calculus, we get
\begin{equation}
G(X_2,P_2Y_2)=G(P_2X_2,P_2Y_2)=G(X_2,Y_2)-\zeta_2(X_2)\zeta_2(Y_2),\,\,\forall\,X_2,Y_2\in\Gamma(V_2).
\label{yII5}
\end{equation}
Let us consider $\{k_{i}\}$ the dual basis of $\{\frac{\partial}{\partial p_{i}}\}$. Then, with respect  to the basis $\{k_{i}\}$ and $\left\{k_{j}\otimes\frac{\partial}{\partial p_i}\right\}$, respectively, $\zeta_2$ and $P_2$ are locally given by
\begin{equation}
\zeta_2=\stackrel{2}{\zeta^{i}}k_{i}\,,\,P_2=\stackrel{2}{P^{j}_i}k_{j}\otimes\frac{\partial}{\partial p_{i}}\,,\,\stackrel{2}{\zeta^i}=\frac{p^i}{K}\,,\,\stackrel{2}{P^{i}_j}=\delta^{i}_j-\frac{p_jp^{i}}{K^2}.
\label{yII6}
\end{equation}
As usual for cotangent bundle of a Cartan space (see Theorem 2.1 from \cite{I-M}), the second vertical Liouville distribution $V_{\mathcal{E}_2}$ is integrable and it defines a foliation on $\mathcal{T}M$, called the \textit{second vertical Liouville foliation} on the big-tangent manifold $\mathcal{T}M$. Also, some geometric properties of the leaves of vertical foliation $V_2$ can be derived via the second vertical Liouville foliation $V_{\mathcal{E}_2}$.

\subsection{Vertical Liouville distribution $V_{\mathcal{E}}$}

In this subsection we unify the concepts presented in the previous subsection and  we define a vertical Liouville distribution on $\mathcal{T}M$ as the complementary orthogonal distribution in $V$ to the line distribution spanned by the  Liouville vector field $\mathcal{E}=\mathcal{E}_1+\mathcal{E}_2$. We prove that this distribution is an integrable one, and also,  we find some geometric properties of both leaves of Liouville distribution and the vertical distribution on the big-tangent manifold $\mathcal{T}M$. Finally, some links between the vertical Liouville foliations $V_{\mathcal{E}_1}$, $V_{\mathcal{E}_2}$ and $V_{\mathcal{E}}$, respectively, are established.

By \eqref{I5}, \eqref{I6} and \eqref{I7} we have
\begin{equation}
G(\mathcal{E},\mathcal{E})=F^2+K^2.
\label{II1}
\end{equation}
Now, by means of $G$ and $\mathcal{E}$, we define the vertical one form $\zeta$ by
\begin{equation}
\zeta(X)=\frac{1}{\sqrt{F^2+K^2}}G(X,\mathcal{E})\,,\,\forall\,X=X_1^{i}(x,y,p)\frac{\partial}{\partial y^{i}}+X^2_i(x,y,p)\frac{\partial}{\partial p_i}\in\Gamma(V).
\label{II2}
\end{equation}

Let us denote by $\left\{\mathcal{E}\right\}$ the line vector bundle over $\mathcal{T}M$ spanned by $\mathcal{E}$ and we define the \textit{vertical Liouville distribution} as the complementary orthogonal distribution $V_{\mathcal{E}}$ to $\left\{\mathcal{E}\right\}$ in $V$ with respect to $G$. Thus, $V_{\mathcal{E}}$ is defined by $\zeta$, that is
\begin{equation}
\Gamma\left(V_{\mathcal{E}}\right)=\{X\in\Gamma(V)\,:\,\zeta(X)=0\}.
\label{II3}
\end{equation}
We get that every vertical vector field $X=X_1^{i}(x,y,p)\frac{\partial}{\partial y^{i}}+X^2_i(x,y,p)\frac{\partial}{\partial p_i}$ can be expressed in the form:
\begin{equation}
X=PX+\frac{1}{\sqrt{F^2+K^2}}\zeta(X)\mathcal{E},
\label{II4}
\end{equation}
where $P$ is the projection morphism of $V$ on $V_{\mathcal{E}}$. 

Also, by direct calculus, we get
\begin{equation}
G(X,PY)=G(PX,PY)=G(X,Y)-\zeta(X)\zeta(Y),\,\,\forall\,X,Y\in\Gamma(V).
\label{II5}
\end{equation}
With respect  to the basis $\{\theta^{i}, k_i\}$ and $\left\{\theta^{j}\otimes\frac{\partial}{\partial y^i}, \theta^{j}\otimes\frac{\partial}{\partial p_i}, k_{j}\otimes\frac{\partial}{\partial y^i}, k_{j}\otimes\frac{\partial}{\partial p_i}\right\}$, respectively, $\zeta$ and $P$ are locally given by
\begin{equation}
\zeta=\zeta_i\theta^{i}+\zeta^{i}k_i\,,\,P=\stackrel{1}{P^{i}_j}\theta^{j}\otimes\frac{\partial}{\partial y^{i}}+\stackrel{2}{P^{j}_i}k_j\otimes\frac{\partial}{\partial p_{i}}+\stackrel{3}{P_{ij}}\theta^{j}\otimes\frac{\partial}{\partial p_{i}}+\stackrel{4}{P^{ij}}k_j\otimes\frac{\partial}{\partial y^{i}},
\label{II6}
\end{equation}
where their local components are expressed by
\begin{equation}
\label{II7}
\zeta_i=\frac{y_i}{\sqrt{F^2+K^2}}\,,\,\zeta^{i}=\frac{p^{i}}{\sqrt{F^2+K^2}},
\end{equation}
\begin{equation}
\label{II8}
\stackrel{1}{P^{i}_j}=\delta^{i}_j-\frac{y_jy^{i}}{F^2+K^2}\,,\,\stackrel{2}{P^{i}_j}=\delta^{i}_j-\frac{p^ip_j}{F^2+K^2}\,,\,\stackrel{3}{P_{ij}}=-\frac{y_jp_i}{F^2+K^2}\,,\,\stackrel{4}{P^{ij}}=-\frac{p^jy^{i}}{F^2+K^2}.
\end{equation}
\begin{remark}
We have the following relations between $\zeta$, $P$, $\zeta_1$, $\zeta_2$, $P_1$ and $P_2$:
\begin{equation}
\label{a1}
\zeta(X)=\frac{F}{\sqrt{F^2+K^2}}\zeta_1(X_1)+\frac{K}{\sqrt{F^2+K^2}}\zeta_2(X_2),
\end{equation} 
\begin{equation}
\label{a2}
P(X)=P_1(X_1)+P_2(X_2)+\frac{1}{F^2+K^2}\left(\frac{\zeta_1(X_1)}{F}-\frac{\zeta_2(X_2)}{K}\right)(K^2\mathcal{E}_1-F^2\mathcal{E}_2),
\end{equation}
for every vertical vector field $X=X_1+X_2=X_1^{i}(x,y,p)\frac{\partial}{\partial y^{i}}+X^2_i(x,y,p)\frac{\partial}{\partial p_i}$.
\end{remark}
\begin{theorem}
\label{integrabil}
The vertical Liouville distribution $V_{\mathcal{E}}$ is integrable and it defines a foliation on $\mathcal{T}M$, called vertical Liouville foliation on the big-tangent manifold $\mathcal{T}M$.
\end{theorem}
\begin{proof}
Follows using an argument similar to that used in \cite{B-F1}. Let $X,Y\in\Gamma\left(V_{\mathcal{E}}\right)$. As $V$ is an integrable distribution on $\mathcal{T}M$, it is sufficient to prove that $[X,Y]$ has no component with respect to $\mathcal{E}$. 

It is easy to see that a vertical vector field $X=X_1^{i}(x,y,p)\frac{\partial}{\partial y^{i}}+X^2_i(x,y,p)\frac{\partial}{\partial p_i}$ is in $\Gamma\left(V_{\mathcal{E}}\right)$ if and only if
\begin{equation}
g_{ij}(x,y)X_1^{i}y^{j}+g^{*ij}(x,p)X^2_ip_j=0.
\label{II9}
\end{equation}
Differentiate ({\ref{II9}}) with respect to $y^k$ we get
\begin{equation}
\frac{\partial g_{ij}}{\partial y^k}X_1^iy^j+g_{ik}X_1^i+g_{ij}\frac{\partial X_1^i}{\partial y^k}y^j+g^{*ij}p_j\frac{\partial X^2_i}{\partial y^{k}}=0\,,\,\forall\,k=1,\ldots,n
\label{II10}
\end{equation}
and taking into account the relation $\frac{\partial g_{ij}}{\partial y^k}y^j=0$ (see \eqref{I6}), one gets
\begin{equation}
g_{ik}X_1^i+g_{ij}y^j\frac{\partial X_1^i}{\partial y^k}+g^{*ij}p_j\frac{\partial X^2_i}{\partial y^{k}}=0\,,\,\forall\,k=1,\ldots,n.
\label{II11}
\end{equation}
Similarly, differentiate ({\ref{II9}}) with respect to $p_k$ we get
\begin{equation}
g_{ij}y^j\frac{\partial X_1^i}{\partial p_k}+g^{*ik}X^2_i+\frac{\partial g^{*ij}}{\partial p_k}X^2_ip_j+g^{*ij}p_j\frac{\partial X^2_i}{\partial p_k}=0\,,\,\forall\,k=1,\ldots,n
\label{II12}
\end{equation}
and taking into account the relation $\frac{\partial g^{*ij}}{\partial p_k}p_j=0$ (see \eqref{I7}), one gets
\begin{equation}
g^{*ik}X^2_i+g_{ij}y^j\frac{\partial X_1^i}{\partial p_k}+g^{*ij}p_j\frac{\partial X^2_i}{\partial p_k}=0\,,\,\forall\,k=1,\ldots,n.
\label{II13}
\end{equation}
Let $X=X_1^{i}(x,y,p)\frac{\partial}{\partial y^{i}}+X^2_i(x,y,p)\frac{\partial}{\partial p_i}$, $Y=Y_1^{j}(x,y,p)\frac{\partial}{\partial y^{j}}+Y^2_j(x,y,p)\frac{\partial}{\partial p_j}\in\Gamma(V)$. Then, by direct calculations using \eqref{II11} and \eqref{II13}, we have
\begin{eqnarray*}
G([X,Y],\mathcal{E}) &=& g_{jk}y^k\left(X_1^{i}\frac{\partial Y_1^j}{\partial y^{i}}-Y_1^{i}\frac{\partial X_1^{j}}{\partial y^{i}}\right)+g^{*ik}p_kX_1^j\frac{\partial Y^2_i}{\partial y^j}-g_{ik}y^kY^2_j\frac{\partial X_1^{i}}{\partial p_j}\\
&=&+g_{ik}y^kX^2_j\frac{\partial Y_1^{i}}{\partial p_j}-g^{*ik}p_kY_1^j\frac{\partial X^2_i}{\partial y^j}+g^{*ik}p_k\left(X^2_j\frac{\partial Y^2_i}{\partial p_j}-Y^2_j\frac{\partial X^2_i}{\partial p_j}\right)\\
&=&-g_{ij}Y_1^{i}X_1^j+g_{ij}X_1^{i}Y_1^j-g^{*ij}Y^2_iX^2_j+g^{*ij}X^2_iY^2_j\\
&=&0
\end{eqnarray*}
which completes the proof.
\end{proof}
\begin{remark}
The proof of Theorem \ref{integrabil} can be also obtained using an argument similar to \cite{B-Mu}. More exactly, if we consider $P(\frac{\partial}{\partial y^{j}})=\stackrel{1}{P^{i}_j}\frac{\partial}{\partial y^{i}}+\stackrel{3}{P_{ij}}\frac{\partial}{\partial p_{i}}$ and $P(\frac{\partial}{\partial p_{j}})=\stackrel{4}{P^{ij}}\frac{\partial}{\partial y^{i}}+\stackrel{2}{P^j_{i}}\frac{\partial}{\partial p_{i}}$, by direct calculus we obtain
\begin{equation}
\label{r1}
P(\frac{\partial}{\partial y^{j}})(\sqrt{F^2+K^2})=P(\frac{\partial}{\partial p_{j}})(\sqrt{F^2+K^2})=0.
\end{equation}
Now, since $V=V_{\mathcal{E}}\oplus\{\mathcal{E}\}$ is integrable, the Lie brackets of vector fields from $V_{\mathcal{E}}$ are given by
\begin{equation}
\label{r3}
\left[P(\frac{\partial}{\partial y^{i}}),P(\frac{\partial}{\partial y^{j}})\right]=A_{ij}^kP(\frac{\partial}{\partial y^{k}})+B_{ijk}P(\frac{\partial}{\partial p_{k}})+C_{ij}\mathcal{E},
\end{equation} 
\begin{equation}
\label{r4}
\left[P(\frac{\partial}{\partial y^{i}}),P(\frac{\partial}{\partial p_{j}})\right]=D_{i}^{jk}P(\frac{\partial}{\partial y^{k}})+E^j_{ik}P(\frac{\partial}{\partial p_{k}})+F_i^j\mathcal{E},
\end{equation}
\begin{equation}
\label{r5}
\left[P(\frac{\partial}{\partial p_{i}}),P(\frac{\partial}{\partial p_{j}})\right]=G^{ijk}P(\frac{\partial}{\partial y^{k}})+H^{ij}_{k}P(\frac{\partial}{\partial p_{k}})+L^{ij}\mathcal{E},
\end{equation}
for some locally defined functions $A_{ij}^k$, $B_{ijk}$, $C_{ij}$, $D_{i}^{jk}$, $E^j_{ik}$, $F_i^j$, $G^{ijk}$, $H^{ij}_{k}$ and $L^{ij}$, respectively. We notice that by the homogeneity condition of $F$ and $K$ we have $\mathcal{E}(\sqrt{F^2+K^2})=\sqrt{F^2+K^2}$. Now, if we apply the vector fields in both sides of formulas \eqref{r3}, \eqref{r4} and \eqref{r5} to the function $\sqrt{F^2+K^2}$ and using \eqref{r1}, we obtain $C_{ij}\sqrt{F^2+K^2}=F_i^j\sqrt{F^2+K^2}=L^{ij}\sqrt{F^2+K^2}=0$. This implies that $C_{ij}=F_i^j=L^{ij}=0$, and then the vertical Liouville distribution $V_{\mathcal{E}}$ is integrable.
\end{remark}

As usual, the Theorem \ref{integrabil}, we may say that the geometry of the leaves of vertical foliation $V$ should be derived from the geometry of the leaves of vertical Liouville foliation $V_\mathcal{E}$ and of integral curves
of $\mathcal{E}$. In order to obtain this interplay, we consider a leaf $L_V$ of $V$ given locally by $x^i=a^i$, $i=1,\ldots,n$, where the $a^i$'s are constants. Then, $g_{ij}(a,y)$ and $g^{*ij}(a,p)$ are
the components of a Riemannian metric $G_{L_V}=G|_{L_V}$ on $L_V$. If we denote by  $\nabla$ the
Levi-Civita connection on $L_V$ with respect to $G_{L_V}$ then its local expression is
\begin{equation}
\nabla_{\frac{\partial}{\partial y^{i}}}\frac{\partial}{\partial y^j}=C^k_{ij}(a,y)\frac{\partial}{\partial y^k}\,,\,\nabla_{\frac{\partial}{\partial y^{i}}}\frac{\partial}{\partial p_j}=0\,,\,\nabla_{\frac{\partial}{\partial p_i}}\frac{\partial}{\partial y^j}=0\,,\,\nabla_{\frac{\partial}{\partial p_{i}}}\frac{\partial}{\partial p_j}=C_k^{ij}(a,p)\frac{\partial}{\partial p_k},
\label{II14}
\end{equation}
where $C^k_{ij}(a,y)=\frac{1}{2}g^{lk}(a,y)\frac{\partial g_{jl}(a,y)}{\partial y^{i}}$ and $C^{ij}_k(a,p)=-\frac{1}{2}g^*_{lk}(a,p)\frac{\partial g^{*jl}(a,p)}{\partial p_i}$.

Contracting $C^k_{ij}(a,y)$ by $y^j$ and $C^{ij}_k(a,p)$ by $p_j$, respectively, we deduce 
\begin{equation}
C^k_{ij}(a,y)y^j=0\,\,,\,\,C_k^{ij}(a,p)p_j=0.
\label{II15}
\end{equation}
In the following lemma we obtain the covariant derivatives with respect to $\nabla$ of $\mathcal{E}$, $\zeta$ and $P$, respectively.
\begin{lemma}
On any leaf $L_V$ of $V$, we have
\begin{equation}
\nabla_X\left(\frac{\mathcal{E}}{\sqrt{F^2+K^2}}\right)=\frac{PX}{\sqrt{F^2+K^2}},
\label{II16}
\end{equation}
\begin{equation}
\left(\nabla_X\zeta\right)Y=\frac{1}{\sqrt{F^2+K^2}}G_{L_V}(PX,PY),
\label{II17}
\end{equation}
and
\begin{equation}
\left(\nabla_XP\right)Y=-\frac{1}{F^2+K^2}\left[G_{L_V}(PX,PY)\mathcal{E}+\sqrt{F^2+K^2}\zeta(Y)PX\right]
\label{II18}
\end{equation}
for any $X,Y\in\Gamma\left(TL_V\right)$.
\end{lemma}
\begin{proof}
We take $X=X_1^i(a,y,p)\frac{\partial}{\partial y^i}+X^2_i(a,y,p)\frac{\partial}{\partial p_i},\,Y=Y_1^j(a,y,p)\frac{\partial}{\partial y^j}+Y^2_j(a,y,p)\frac{\partial}{\partial p_j}\in\Gamma\left(TL_V\right)$ and the relation \eqref{II16} follows by:
\begin{eqnarray*}
\nabla_X\left(\frac{\mathcal{E}}{\sqrt{F^2+K^2}}\right)&=&\frac{X_1^i}{\sqrt{F^2+K^2}}\left(\delta_i^j-\frac{y^jy_i}{F^2+K^2}\frac{\partial}{\partial y^j}-\frac{p_jy_i}{F^2+K^2}\frac{\partial}{\partial p_j}\right)\\
&&+ \frac{X^2_i}{\sqrt{F^2+K^2}}\left(\delta_j^i-\frac{p_jp^{i}}{F^2+K^2}\frac{\partial}{\partial p_j}-\frac{y^jp^{i}}{F^2+K^2}\frac{\partial}{\partial y^j}\right)\\
&=&\frac{1}{\sqrt{F^2+K^2}}\left(X_1^{i}\stackrel{1}{P^j_i}\frac{\partial}{\partial y^j}+X_1^{i}\stackrel{3}{P_{ji}}\frac{\partial}{\partial p_j}+X^2_i\stackrel{4}{P^{ji}}\frac{\partial}{\partial y^j}+X^2_i\stackrel{2}{P^i_j}\frac{\partial}{\partial p_j}\right)\\
&=&\frac{PX}{\sqrt{F^2+K^2}}.
\end{eqnarray*}
For the relation \eqref{II17} we have
\begin{eqnarray*}
\left(\nabla_X\zeta\right)Y&=&X(\zeta(Y))-\zeta\left(\nabla_XY\right)\\
&=& X_1^iY_1^j\frac{\partial\zeta_j}{\partial y^{i}}+X_1^iY^2_j\frac{\partial\zeta^j}{\partial y^{i}}+X^2_iY_1^j\frac{\partial\zeta_j}{\partial p_{i}}+X^2_iY^2_j\frac{\partial\zeta^j}{\partial p_{i}}\\
&=&\frac{X_1^iY_1^j}{\sqrt{F^2+K^2}}\left(g_{ij}-\frac{y_iy_j}{F^2+K^2}\right)-\frac{X_1^iY^2_jp^jy_i}{(F^2+K^2)\sqrt{F^2+K^2}}\\
&&-\frac{X^2_iY_1^jy_jp^{i}}{(F^2+K^2)\sqrt{F^2+K^2}}+\frac{X^2_iY^2_j}{\sqrt{F^2+K^2}}\left(g^{*ij}-\frac{p^jp^{i}}{F^2+K^2}\right).
\end{eqnarray*}
On the other hand we have
\begin{eqnarray*}
G_{L_V}(PX,PY)&=&G_{L_V}(X,Y)-\zeta(X)\zeta(Y)\\
&=&X_1^iY_1^jg_{ij}+X^2_iY^2_jg^{*ij}-\frac{(X_1^{i}y_i+X^2_{i}p^i)(Y_1^{j}y_j+Y^2_{j}p^j)}{F^2+K^2}
\end{eqnarray*}
and the relation \eqref{II17} follows easy.

The relation \eqref{II18} folows using \eqref{II4}, \eqref{II16} and \eqref{II17}. 
\end{proof}

\begin{theorem}
\label{t2}
Let $(M,F)$ be a $n$-dimensional Finsler space and $L_V$, $L_{V_{\mathcal{E}}}$ and $\gamma$ be a leaf of $V$, a leaf of $V_{\mathcal{E}}$ that lies in $L_V$, and an integral curve of $\frac{\mathcal{E}}{\sqrt{F^2+K^2}}$, respectively. Then the following assertions are valid:
\begin{enumerate}
\item[i)] $\gamma$ is a geodesic of $L_V$ with respect to $\nabla$.
\item[ii)] $L_{V_{\mathcal{E}}}$ is totally umbilical immersed in $L_V$.
\item[iii)] $L_{V_{\mathcal{E}}}$ lies in the generalized indicatrix $I_a=\{(y,p)\in T_aM^0\oplus T_a^*M^0\,:\,F^2(a,y)+K^2(a,p)=1\}$ and has constant mean curvature equal to $-1$.
\end{enumerate}
\end{theorem}
\begin{proof}
Replace $X$ by $\frac{\mathcal{E}}{\sqrt{F^2+K^2}}$ in \eqref{II16} and we obtain i). Taking into account that $\frac{\mathcal{E}}{\sqrt{F^2+K^2}}$ is the unit normal vector field of $L_{V_{\mathcal{E}}}$, the second fundamental form $B$ of $L_{V_{\mathcal{E}}}$ as a hypersurface of $L_V$ is given by
\begin{equation}
\label{II19}
B(X,Y)=\frac{1}{\sqrt{F^2+K^2}}G_{L_V}\left(\nabla_XY,\mathcal{E}\right)\,,\,\forall\,X,Y\in\Gamma\left(TL_{V_{\mathcal{E}}}\right).
\end{equation}
On the other hand, by using \eqref{II16} and taking into account that $G_{L_V}$ is parallel with respect to $\nabla$, we deduce that
\begin{equation}
\label{II20}
G_{L_V}\left(\nabla_XY,\mathcal{E}\right)=-G_{L_V}(X,Y)\,,\,\forall\,X,Y\in\Gamma\left(TL_{V_{\mathcal{E}}}\right).
\end{equation}
Hence,
\begin{equation}
\label{II21}
B(X,Y)=-\frac{1}{\sqrt{F^2+K^2}}G_{L_V}(X,Y), \,\forall\,X,Y\in\Gamma\left(TL_{V_{\mathcal{E}}}\right),
\end{equation}
that is, $L_{V_{\mathcal{E}}}$ is totally umbilical immersed in $L_V$. Now, we have 
\begin{equation}
\label{II22}
\frac{g_{ij}y^{i}}{\sqrt{F^2+K^2}}+\frac{g^{*ij}p_i}{\sqrt{F^2+K^2}}=\frac{\partial\sqrt{F^2+K^2}}{\partial y^j}+\frac{\partial\sqrt{F^2+K^2}}{\partial p_j}
\end{equation}
which says that $\frac{\mathcal{E}}{\sqrt{F^2+K^2}}$ is a unit normal vector field for both $L_{V_{\mathcal{E}}}$ and the component $I_a$. Thus, $L_{V_{\mathcal{E}}}$ lies in $I_a$ and $F^2(a,y)+K^2(a,p)=1$ at any point $(y,p)\in L_{V_{\mathcal{E}}}$. Then \eqref{II21} becomes
\begin{equation}
\label{II23}
B(X,Y)=-G_{L_V}(X,Y),\,\forall\,X,Y\in\Gamma\left(TL_{V_{\mathcal{E}}}\right)
\end{equation}
which implies that
\begin{equation}
\label{II24}
\frac{1}{2n-1}\sum_{i=1}^{2n-1}\varepsilon_iB(E_i,E_i)=-1,
\end{equation}
where $\{E_i\}$ ia an orthonormal frame field on $L_{V_{\mathcal{E}}}$ of signature $\{\varepsilon_i\}$. Hence, the mean curvature of $L_{V_{\mathcal{E}}}$ is $-1$ which completes the proof.
\end{proof}
\begin{theorem}
\label{t3}
Let $(M,F)$ be a $n$-dimensional Finsler space and $L_V$ be a leaf of the vertical foliation $V$. Then the sectional curvature of any nondegenerate plane section on $L_V$ which contain the vertical Liouville vector field $\mathcal{E}$ is equal to zero.
\end{theorem}
\begin{proof}
Denote by $R_{L_V}$ the curvature tensor field of $\nabla$ on $L_V$. Then, by using \eqref{II16} and \eqref{II18}, we obtain
\begin{equation}
\label{II25}
R_{L_V}\left(X,\mathcal{E}\right)\mathcal{E}=-\left(1-\frac{\mathcal{E}(\sqrt{F^2+K^2})}{\sqrt{F^2+K^2}}\right)PX
\end{equation}
for every vector field $X$ on $L_V$. Now, taking into account $\mathcal{E}(\sqrt{F^2+K^2})=\sqrt{F^2+K^2}$, the sectional curvature of a plane section $\{X,\mathcal{E}\}$ vanishes on $L_V$.
\end{proof}
\begin{corollary}
Let $(M,F)$ be a $n$-dimensional Finsler space. Then there exist no leaves of $V$ which are positively or negatively curved.
\end{corollary}

Finally, let us study certain relations between the vertical Liouville foliations $V_{\mathcal{E}_1}$, $V_{\mathcal{E}_2}$ and $V_{\mathcal{E}}$, respectively.

We notice that we have the following decompositions of the vertical distribution:
\begin{equation}
\label{a3}
V=V_{\mathcal{E}_1}\oplus V_{\mathcal{E}_2}\oplus\{\mathcal{E}_1\}\oplus\{\mathcal{E}_2\}\,\,{\rm and}\,\,V=V_{\mathcal{E}}\oplus\{\mathcal{E}\}.
\end{equation}
Taking into account that $[\stackrel{1}{P^{i}_j}\frac{\partial}{\partial y^{i}},\stackrel{2}{P^k_l}\frac{\partial}{\partial p_l}]=0$ and $[\mathcal{E}_1,\mathcal{E}_2]=0$ we get that both distributions $V_{\mathcal{E}_1}\oplus V_{\mathcal{E}_2}$ and $\{\mathcal{E}_1\}\oplus\{\mathcal{E}_2\}$ are integrable. Evidently, $\{\mathcal{E}\}\subset\{\mathcal{E}_1\}\oplus\{\mathcal{E}_2\}$ and by \eqref{a1} we have also $V_{\mathcal{E}_1}\oplus V_{\mathcal{E}_2}\subset V_{\mathcal{E}}$. Thus, we have the following vertical subfoliations on $\mathcal{T}M$:
\begin{equation}
\label{a4}
\{\mathcal{E}\}\subset\{\mathcal{E}_1\}\oplus\{\mathcal{E}_2\}\subset V\,,\,V_{\mathcal{E}_1}\oplus V_{\mathcal{E}_2}\subset V_{\mathcal{E}}\subset V.
\end{equation} 
The relations \eqref{a3} says that $\{\mathcal{E}\}$ and $V_{\mathcal{E}_1}\oplus V_{\mathcal{E}_2}$ have the same orthogonal complement in $\{\mathcal{E}_1\}\oplus\{\mathcal{E}_2\}$ and in $V_{\mathcal{E}}$, respectively. It is a line distribution $\{\mathcal{E}^\prime\}$, where $\mathcal{E}^\prime=K^2\mathcal{E}_1-F^2\mathcal{E}_2$, see \eqref{a2} (or by direct calculations in $G(\alpha_1\mathcal{E}_1+\alpha_2\mathcal{E}_2,\mathcal{E})=0$ it results $\alpha_1=K^2$ and $\alpha_2=-F^2$). Thus
\begin{equation}
\label{a5}
\{\mathcal{E}_1\}\oplus\{\mathcal{E}_2\}=\{\mathcal{E}\}\oplus\{\mathcal{E}^\prime\}\,,\,V_{\mathcal{E}}=V_{\mathcal{E}_1}\oplus V_{\mathcal{E}_2}\oplus\{\mathcal{E}^\prime\}.
\end{equation}

\begin{proposition}
The leaves of the foliation $\{\mathcal{E}_1\}\oplus\{\mathcal{E}_2\}$ are totally geodesic submanifolds of the leaves of vertical foliation $V$.
\end{proposition}
\begin{proof}
Follows easily taking into account that $\nabla_{\mathcal{E}_1}\mathcal{E}_1=\mathcal{E}_1\,,\,\nabla_{\mathcal{E}_1}\mathcal{E}_2=\nabla_{\mathcal{E}_2}\mathcal{E}_1=0\,,\,\nabla_{\mathcal{E}_2}\mathcal{E}_2=\mathcal{E}_2$.
\end{proof}

Also by direct calculus we obtain $\nabla_{\mathcal{E}^\prime}\mathcal{E}^\prime=-K^2F^2\mathcal{E}+(K^2-F^2)\mathcal{E}^\prime\notin\Gamma(\{\mathcal{E}^\prime\})$, which leads to
\begin{proposition}
If $\gamma$ is an integral curve of $\mathcal{E}^\prime$ then it is not a geodesic of a leaf of vertical foliation $V$.
\end{proposition}

A natural question is if between the foliations $V_{\mathcal{E}_1}\oplus V_{\mathcal{E}_2}$ and $V_{\mathcal{E}}$ exists certain relations. Although, the leaves of $V_{\mathcal{E}_1}$ are totally umbilical submanifolds of the leaves of $V_1$, the leaves of $V_{\mathcal{E}_2}$ are totally umbilical submanifolds of the leaves of $V_2$ and the leaves of $V_{\mathcal{E}}$ are totally umbilical submanifolds of the leaves of $V$, we have
\begin{theorem}
The leaves of $V_{\mathcal{E}_1}\oplus V_{\mathcal{E}_2}$ are not totally umbilical submanifolds of the leaves of $V_{\mathcal{E}}$.
\end{theorem}
\begin{proof}
Taking into account that $\frac{\mathcal{E}^\prime}{FK\sqrt{F^2+K^2}}$ is the unit normal vector field of $L_{V_{\mathcal{E}_1}\oplus V_{\mathcal{E}_2}}$, the second fundamental form $B^\prime$ of $L_{V_{\mathcal{E}_1}\oplus V_{\mathcal{E}_2}}$ as hypersurface of $L_{V_{\mathcal{E}}}$ is given by
\begin{equation}
\label{a7}
B^\prime(X^\prime,Y^\prime)=\frac{1}{FK\sqrt{F^2+K^2}}G_{L_V}\left(\nabla_{X^\prime}Y^\prime,\mathcal{E}^\prime\right)\,,\,\forall\,X^\prime,Y^\prime\in\Gamma\left(TL_{V_{\mathcal{E}_1}\oplus V_{\mathcal{E}_2}}\right).
\end{equation}
Taking into account that $G_{L_V}$ is parallel with respect to $\nabla$, we deduce that
\begin{equation}
\label{a8}
G_{L_V}\left(\nabla_{X^\prime}Y^\prime,\mathcal{E}^\prime\right)=-G_{L_V}(Y^\prime,\nabla_{X^\prime}\mathcal{E}^\prime)\,,\,\forall\,X^\prime,Y^\prime\in\Gamma\left(TL_{V_{\mathcal{E}_1}\oplus V_{\mathcal{E}_2}}\right).
\end{equation}
Now, let us take $X^{\prime}=P_1(X_1)+P_2(X_2)$ and $Y^\prime=P_1(Y_1)+P_2(Y_2)$ for every $X_1,Y_1\in\Gamma(V_1)$ and $X_2,Y_2\in\Gamma(V_2)$. Then by direct calculus we get 
\begin{equation}
\label{a9}
\nabla_{X^\prime}\mathcal{E}^\prime=K^2P_1(X_1)-F^2P_2(X_2).
\end{equation}
Thus the relation \eqref{a7} becomes
\begin{equation}
\label{a10}
B^\prime(X^\prime,Y^\prime)=\frac{-1}{FK\sqrt{F^2+K^2}}G_{L_V}\left(K^2P_1(X_1)-F^2P_2(X_2),Y^\prime\right)\neq \lambda G_{L_V}(X^\prime,Y^\prime),
\end{equation}
that is, $L_{V_{\mathcal{E}_1}\oplus V_{\mathcal{E}_2}}$ is not totally umbilical immersed in $L_{V_{\mathcal{E}}}$.
\end{proof}

\noindent
Cristian Ida\\
Department of Mathematics and Computer Science\\
University Transilvania of Bra\c{s}ov\\
Address: Bra\c{s}ov 500091, Str. Iuliu Maniu 50, Rom\^{a}nia\\
email: \textit{cristian.ida@unitbv.ro}\\

\noindent 
Paul Popescu\\
Department of Applied Mathematics\\
 University of Craiova\\
Address: Craiova, 200585,  Str. Al. Cuza, No. 13,  Rom\^{a}nia\\
 email:\textit{paul$_{-}$p$_{-}$popescu@yahoo.com}

\smallskip

\end{document}